\documentclass[12pt]{amsart}

\usepackage{amssymb}

\newcommand{\diag}{\mathop{\rm diag}}

\newcommand{\Cof}{\mathop{\rm Cof}}
\newcommand{\cof}{\mathop{\rm cof}}

\newtheorem{theorem}{\bf  Theorem}
\newtheorem{lemma}{\bf  Lemma}

\begin{document}


\begin{center}
Spectral Factorization of Rank-Deficient Polynomial
Matrix-Functions
\\[5mm]

 L.~Ephremidze
  and  E.~Lagvilava
        \end{center}

\vskip+0.5cm

 \noindent {\small {\bf Abstract.} A spectral factorization theorem
is proved for polynomial rank-deficient matrix-functions. The
theorem is used to construct paraunitary matrix-functions with
first rows given.}

 \vskip+0.2cm\noindent  {\small {\bf Keywords:} Matrix spectral factorization,
paraunitary matrix-polynomials, wavelets}

\vskip+0.2cm \noindent  {\small {\bf  AMS subject classification
(2010):} 47A68,
 42C40}

\vskip+0.5cm

Wiener's spectral factorization theorem [12], [4] for polynomial
matrix-functions asserts that if
\begin{equation}
 S(z)=\sum_{n=-N}^N C_nz^n
\end{equation}
is an $m\times m$ matrix-function ($C_n\in \mathbb{C}^{m \times
m}$ are matrix coefficients) which is positive definite for a.a.
$z\in \mathbb{T}$, $\mathbb{T}:=\{z\in \mathbb{C}: |z|=1\}$, then
it admits a factorization
\begin{equation}
 S(z)=S^+(z)S^-(z) = \sum_{n=0}^N A_nz^n\cdot \sum_{n=0}^N
 A^*_nz^{-n},\;\;\;\;\;\;z\in\mathbb{C}\backslash\{0\},
\end{equation}
where $S^+$ is an $m\times m$ polynomial  matrix-function which is
nonsingular inside $\mathbb{T}$, $\det\,S^+(z)\not=0$ when
$|z|<1$, and $S^-$ is its adjoint, $A_n^*={\overline{A}_n}^T$,
$n=0,1,\ldots,N$. (Respectively, $S^-$ is analytic and nonsingular
outside $\mathbb{T}$.) $S^+$ is unique up to a constant right
unitary multiplier.

The factorization (2) is also known under the name of
matrix-valued Fej\'{e}r-Riesz theorem and its simple proof is
provided in [2]. Various practical applications of this theorem in
system analysis [6] and wavelet design [1] are widely recognized.

In the present paper we consider rank-deficient matrix polynomials
and prove the corresponding spectral factorization theorem for
them:

\begin{theorem} Let $S(z)$ be an $m\times m$ $($trigonometric$)$
polynomial matrix-function $(1)$ of order $N$
$(${}$C_N=C_{-N}^*\not=\mathbf{0}${}$)$ which is nonnegative
definite and of rank $k\leq m$ for a.a. $z\in \mathbb{T}$. Then
there exists a unique $($up to a $k\times k$ unitary matrix right
multiplier$)$ $m\times k$ matrix-polynomial $S^+(z)=\sum_{n=0}^N
A_nz^n$, $A_n\in \mathbb{C}^{m \times k}$ of order $N$
$(${}$A_N\not=\mathbf{0}${}$)$, which is of full rank $k$ for each
$z$ inside $\mathbb{T}$, such that $(2)$ holds.
\end{theorem}

{\em Remark.} If we require of $S^+$ to be just a rational
matrix-function analytic inside $\mathbb{T}$ and drop the
uniqueness from the condition, then the theorem can be obtained in
a standard algebraic manner (see [9]). Hence, as we will see
below, the proof of the theorem provides a simple proof of the
same theorem for the full rank case, $k=m$, as well. This proof is
even more elementary as compared  with the one given in [2] since
it avoids an application of the Hardy space theory.

\smallskip

Prior to proving the theorem, we make some simple observations on
adjoint functions and prove Lemma 1 on paraunitary
matrix-functions. We do not claim that this lemma is new, but
include its proof for the sake of completeness.

If $f$ is an analytic ${m \times k}$ matrix-function in
$\mathbb{C}\backslash \{z_1,z_2,\ldots,z_n\}$, then its adjoint
$f^*(z)=\overline{f\left(1/\overline{z}\right)}^T$ is an analytic
${k \times m}$ matrix-function in $\mathbb{C}\backslash
\{z_1^*,z_2^*,\ldots,z_n^*\}$, $z^*:=1/\overline{z}$,
$\infty^*=0$. Obviously, if $f$ is analytic inside $\mathbb{T}$,
then $f^*$ is analytic outside $\mathbb{T}$ (including infinity).
Namely, if $f_{ij}(e^{i\theta})\in L_1^+(\mathbb{T})$, ($f_{ij}$
is the $ij$th entry of $f$), then $f_{ji}^*(e^{i\theta})\in
L_1^-(\mathbb{T})$, where $L_1^+(\mathbb{T})$
\big($L_1^-(\mathbb{T})$\big) is the set of integrable functions
defined on $\mathbb{T}$ which have Fourier coefficients with
negative (positive) indices equal to zero. Since $f$ is uniquely
determined by its values on $\mathbb{T}$, and
$f^*(z)=\overline{f(z)}^T=(f(z))^*$ for $|z|=1$, usual relations
for adjoint matrix-functions, like $(fg)^*(z)=g^*(z)f^*(z)$ and
$(f^{-1})^*(z)=(f^*)^{-1}(z)$, etc., are valid.

Note that if $f$ is a rational $m\times m$ matrix-function, $f\in
\mathcal{R}^{m\times m}$, then
\begin{equation}
 \left[f(e^{i\theta})f^*(e^{i\theta})\right]_{ii}\in L_\infty(\mathbb{T})
 \Longrightarrow f_{ij} \text{ are free of poles on }\mathbb{T},\;
j=1,2,\ldots,m,
\end{equation}
($L_\infty(\mathbb{T})$ stands for the set of bounded functions)
since $\left[f(z)f^*(z)\right]_{ii}=\sum_{j=1}^m|f_{ij}(z)|^2$
when $|z|=1$.

$U\in \mathcal{R}^{m\times m}$ is called paraunitary if
\begin{equation}
 U(z)U^*(z)=I_m\;\;\;\text{ in the domain of $U$ and $U^*$},
\end{equation}
where $I_m$ stands for the $m$-dimensional unit matrix. Note that
$U(z)$ is a usual unitary matrix for each $z\in \mathbb{T}$, since
$U^*(z)=\overline{U(z)}^T=(U(z))^*$ when $|z|=1$ and,
consequently,
\begin{equation}
\overline{U(z)}^T=U^{-1}(z),\;\;\;z\in \mathbb{T}.
\end{equation}

\begin{lemma} If $U\in \mathcal{R}^{m\times m}$ is paraunitary
and analytic inside $\mathbb{T}$ $($its entries are free of poles
inside $\mathbb{T}${$)$}, and $U^{-1}\in \mathcal{R}^{m\times m}$
is analytic inside $\mathbb{T}$ as well, then $U$ is a constant
unitary matrix.
\end{lemma}

\begin{proof}
The equation (4) implies that $U_{ij}(z)$, $1\leq i,j\leq m$, are
free of poles on $\mathbb{T}$ (see (3)). Since
$U_{ij}(e^{i\theta})\in L_1^+(\mathbb{T})$ and $ L_1^+(\mathbb{T})
\ni U^{-1}_{ji}(e^{i\theta})=\overline{U_{ij}(e^{i\theta})}$ (see
(5)), we have $U_{ij}(e^{i\theta})\in L_1^+(\mathbb{T})\cap
L_1^-(\mathbb{T})$. Thus $U_{ij}(z)$ is constant for a.a. $z\in
\mathbb{T}$, and hence everywhere in the complex plane.
\end{proof}

\smallskip

{\em Proof of Theorem $1$.} Since $S$ is nonnegative definite on
the unit circle, we have $S^*(z)=S(z)$,
$z\in\mathbb{C}\backslash\{0\}$.

Observe that every polynomial matrix-function always has a
constant rank in its domain except for a finite number of points.
Without loss of generality, we can assume that the $k\times k$
left-upper submatrix of $S$, denoted by $S_{00}$, has the full
rank $k$ (a.e.) so that $S$ has the block matrix form
$$
S(z)=\left(\begin{matrix} S_{00}(z)&S_{01}(z)\\S_{10}(z)&S_{11}(z)
\end{matrix}\right),
$$
where $S_{01}$, $S_{10}=S^*_{01}$ and $S_{11}$ are
matrix-functions of dimensions $k\times(m-k)$, $(m-k)\times k$,
and $(m-k)\times(m-k)$, respectively. Since every $k+1$ rows
(columns) of $S(z)$ are linearly dependent, we have
\begin{equation}
S_{10}(z)S^{-1}_{00}(z)S_{01}(z)=S_{11}(z)\;\;\;\;\;\;\text{(a.e.)}.
\end{equation}
Let
\begin{equation}
S_{00}(z)=S^+_{00}(z)S^-_{00}(z)=S^+_{00}(z)(S^+_{00})^*(z)
\end{equation}
be the polynomial spectral factorization of $S_{00}$ which exists
by virtue of the matrix-valued Fej\'{e}r-Riesz theorem. Define
$$
\sigma_{10}(z):=S_{10}(z)(S^-_{00}(z))^{-1}
$$
and let $S_0$ have the block matrix form
$$
S_0(z)=\left(\begin{matrix} S^+_{00}(z)\\ \sigma_{10}(z)
\end{matrix}\right).
$$
Then $S_0^*(z)= \left[\begin{matrix} S^-_{00}(z)&
(S^+_{00}(z))^{-1}S_{01}(z) \end{matrix}\right]$ and, taking (6)
into account, one can directly check that
\begin{equation}
 S(z)=S_0(z)S_0^*(z).
\end{equation}
Since $S^+_{00}$ is a polynomial matrix-function, $S_0$ is a
rational matrix-function, however it  might not be analytic inside
$\mathbb{T}$. If $s_{ij}$ is the $ij$th entry of $S_0$ with a pole
at $a$ inside $\mathbb{T}$, then we can multiply $S_0$ by the
unitary matrix-function $U(z)=\diag[1,\ldots,u(z),\ldots,1]$,
where $u(z)=(z-a)/(1-\overline{a}z)$ is the $jj$th entry of
$U(z)$, so that the $ij$th entry of the product $S_0(z)U(z)$ will
not have a pole at $a$ any longer keeping the factorization (8):
$$
(S_0U)(z)(S_0U)^*(z)=S_0(z)S_0^*(z)=S(z).
$$
In the same way one can remove every pole of the entries of $S_0$
at points inside $\mathbb{T}$. Thus  $S$ can be represented as a
product
\begin{equation}
 S(z)=S^+_0(z)S_0^-(z),
\end{equation}
where $S^+_0$ is a rational matrix-function which is analytic
inside $\mathbb{T}$, and $S_0^-(z)$ is its adjoint. Note that
$S^+_0(z)$ remains of full rank $k$ for each $z\in \mathbb{T}$
except possibly a finite number of points.

Now, it might happen so that $S_0^+$ is not of full rank $k$
inside $\mathbb{T}$ everywhere. If $|a|<1$ and $rank\,
S_0^+(a)<k$, then there exists a unitary matrix $U$ such that the
product $S_0^+(a)U$ has all 0's in the first column. Hence $a$ is
a zero of every entry of the first column of the matrix-function
$S_0^+(z)U$ and the product $S_1^+(z):=
S_0^+(z)U\diag[u(z),1,\ldots,1]$, where
$u(z)=(1-\overline{a}z)/(z-a)$, remains analytic inside
$\mathbb{T}$. While the factorization (9) remains true replacing
$S_0^+$ and $S_0^-$ by $S_1^+$ and $S_1^-$, respectively, the
minors of $S_1^+$ will have less zeros inside $\mathbb{T}$ than
the minors of $S_0^+$. Thus, continuing this process if necessary,
we get the factorization
\begin{equation}
 S(z)=S^+(z)S^-(z),
\end{equation}
where $S^+$ is a rational matrix-function which is analytic and of
full rank $k$ inside $\mathbb{T}$.

Now let us show that $S^+$ is in fact a polynomial matrix-function
of order $N$. It suffices to show that $z^NS^-(z)$ is analytic
inside $\mathbb{T}$. Indeed, since $S^+$ does not have poles on
$\mathbb{T}$ (see (10) and (3)), $z^NS^-(z)$ should be an analytic
(on the whole $\mathbb{C}$) rational matrix-function in this case,
and therefore a polynomial.

It follows form (10) that
\begin{equation}
z^NS^-(z)=\left((S^+(z))^TS^+(z)\right)^{-1}\cdot (S^+(z))^T\cdot
z^NS(z)
\end{equation}
and $z^NS^-(z)$ is analytic inside $\mathbb{T}$ since each of the
three factors on the right-hand side of (11) is such.

To complete the proof of the theorem, it remains to show that the
factorization (2) is unique, i.e. if
$$
 S(z)=S_1^+(z)S_1^-(z)
$$
where $S_1^+$ is a $m\times k$ polynomial matrix-function which
has the full rank $k$  inside $\mathbb{T}$, then
$$
S_1^+(z)=S^+(z)U
$$
for some $k\times k$ (constant) unitary matrix  $U$.

Since $S^+(z)$ is of the full rank  $k$ for each $z\in\mathbb{C}$
except for some finite number of singular points, there exists a
matrix-function $U(z)$ such that
\begin{equation}
S_1^+(z)=S^+(z)U(z)
\end{equation}
Thus $U(z)$ can be determined by the equation
$$
U(z)=\left((S^+)^T(z)S^+(z)\right)^{-1} (S^+)^T(z) S_1^+(z)
$$
as a rational function in $\mathbb{C}$. Note that $U(z)$ is
analytic inside $\mathbb{T}$, and since $S^+$ and $S_1^+$
participate symmetrically in the theorem, $U^{-1}(z)$ is analytic
inside $\mathbb{T}$ as well.

Due to Lemma 1, it remains to show that $U\in \mathcal{R}^{k\times
k}$ is  a paraunitary matrix-function. From the equation (12), one
can determine $U(z)$ as
$$
U(z)=\left(S^-(z)S^+(z)\right)^{-1} S^-(z) S_1^+(z)
$$
and, consequently,
$$
U^*(z)= S_1^-(z) S^+(z)\left(S^-(z)S^+(z)\right)^{-1}.
$$
Hence
\begin{gather*}
U(z)U^*(z)=\left(S^-(z)S^+(z)\right)^{-1} S^-(z) S_1^+(z)\cdot
S_1^-(z) S^+(z)\left(S^-(z)S^+(z)\right)^{-1}=\\
=\left(S^-(z)S^+(z)\right)^{-1} S^-(z) S^+(z)S^-(z)
S^+(z)\left(S^-(z)S^+(z)\right)^{-1}=I_k\,.
\end{gather*}
The proof of the theorem is complete.

\smallskip

{\em Remark.} As one can observe, the above proof of the existence
of $S^+$ is constructive. There are several classical algorithms
to perform the factorization (7) numerically in the full rank case
(a new efficient algorithm of such type is proposed in [5]).
Further using the steps described in the proof, one can compute
$S^+$ numerically.

\smallskip

Our next theorem illustrates one of the   applications of Theorem
1  in some areas of signal processing. Namely, $m\times m$
paraunitary matrix-functions
\begin{equation}
 U(z)=\sum_{n=0}^N
 \rho_nz^n=\big[u_{ij}(z)\big]_{i,j=\overline{1m}}\;\,,\;\;\;\;
 \rho_n\in\mathbb{C}^{m\times m},
\end{equation}
defined by (4) play an important role in the theory of wavelets
and multirate filter banks [8] where they are known under
different names, for example, lossless systems [11], perfect
reconstruction $m$-filters [7], paraunitary $m$-channel filters
[10], and so on. The positive integers $m$ and $N$ are called the
{\em size} and the {\em length} of $U$, respectively. Sometimes,
the first row of a matrix-function $U$ is called the {\em low-pass
filter}, and the remaining rows are called the {\em high-pass
filters}. Theorem 2 allows us to find the set of matching
high-pass filters to each low-pass filter. First we give a simple
proof of the following lemma which provides additional information
about structures of paraunitary matrix-polynomials.

\begin{lemma}
{\rm (cf. [8, Lemma 4.13])} Let $(13)$ be a paraunitary
matrix-polynomial of length $N$ $(\rho_N\not=\mathbf{0})$. Then
\begin{equation}
\det U(z)=c\cdot z^k, \text{ where } |c|=1, \text{ and } k\geq N.
\end{equation}
\end{lemma}

\begin{proof}
Since $\det U(z)\cdot\det U^*(z)=1$ and $\det U(z)$ is a
polynomial, it follows that $\det U(z)=cz^k$ for some nonnegative
integer $k$. We have
\begin{equation}
\sum_{n=0}^N
 \rho^*_nz^{-n}= U^*(z)=U^{-1}(z)=\frac1{\det U(z)}\big(\Cof U(z)\big)^T=
 cz^{-k} \big(\Cof U(z)\big)^T.
\end{equation}
Therefore $k\geq N$, since $\Cof U(z)$ is a polynomial
matrix-function and $\rho^*_N$ is not the zero matrix.
\end{proof}

{\em Remark.} The positive integer $k$ in (14) is called the {\em
degree} of $U$. Generically, a paraunitary matrix-polynomial $U$
of length $N$ has the same degree $N$, although in some specific
cases the degree is more than $N$.

\smallskip

The following theorem was first established in [3] by a different
method, however the presented approach gives a new insight to the
problem.

\begin{theorem} For any polynomial vector-function
\begin{equation}
 U_1(z)=\big[u_{11}(z),u_{12}(z),\ldots,u_{1m}(z)\big],
\end{equation}
$u_{1j}(z)=\sum_{n=0}^N \alpha_{jn}z^n$, $j=1,2,\ldots,m$, of
length $N$  $\big(\sum_{j=1}^m |\alpha_{jn}|>0\big)$ which is of
unit norm on $\mathbb{T}$
\begin{equation}
 \|U_1(z)\|_{\mathbb{C}^m}^2=\sum_{j=1}^m |u_{1j}(z)|^2=1,\;\;\;\;z\in \mathbb{T},
\end{equation}
there exists a unique $($up to a constant left multiplier of the
block matrix form $\left(\begin{matrix}
1&0\\0&U\end{matrix}\right)$, where $U$ is a $(m-1)\times (m-1)$
unitary matrix$)$ paraunitary matrix-function $U(z)$ $($of size
$m$ and length $N${}$)$, with determinant
 $cz^N$, $|c|=1$, whose first row is equal to $ (16)$.
\end{theorem}

\begin{lemma} Let $\mathbf{v}=(v_1,v_2,\ldots,v_m)^T\in
\mathbb{C}^m$ be a vector of unit norm,
$\|\mathbf{v}\|^2=\mathbf{v}^*\mathbf{v}= \sum_{j=1}^m|v_j|^2=1$.
Then $I_m-\mathbf{v}\mathbf{v}^*$ is a nonnegative definite
matrix,
\begin{equation}
I_m-\mathbf{v}\mathbf{v}^*\geq 0,
\end{equation}
and
\begin{equation}
rank\,\big(I_m-\mathbf{v}\mathbf{v}^*\big)=m-1.
\end{equation}
\end{lemma}

\begin{proof}
For each column vector $\mathbf{x}\in \mathbb{C}^m$, we have
$$
\mathbf{x}^*(I_m-\mathbf{v}\mathbf{v}^*)\mathbf{x}=\|\mathbf{x}\|^2-
|\mathbf{x}^*\mathbf{v}|^2\geq \|\mathbf{x}\|^2-
\|\mathbf{x}^*\|^2\|\mathbf{v}\|^2=\|\mathbf{x}\|^2-\|\mathbf{x}^*\|^2=0.
$$
Hence (18) holds and
$\mathbf{x}^*(I_m-\mathbf{v}\mathbf{v}^*)\mathbf{x}=0$ if and only
if $\mathbf{x}=\alpha \mathbf{v}$ for some $\alpha\in \mathbb{C}$.
Thus (19) holds as well.
\end{proof}

{\em Proof of Theorem $2$.} Due to Lemma 3 and the property (17),
the matrix-function
\begin{equation}
S(z)=I_m-U_1^T(z)(U_1^T)^*(z)
\end{equation}
is positive definite and of rank $m-1$ for each $z\in \mathbb{T}$.
(Note that the order of $S$ is less than or equal to $N$.) Hence,
by virtue of Theorem 1, there exists an $m\times(m-1)$
matrix-function $S^+(z)$ of full rank $m-1$, for each $z$ inside
$\mathbb{T}$, such that (2) holds. Consequently,
$$
\left[\begin{matrix} U_1^T(z) & S^+(z)\end{matrix}\right]
\left[\begin{matrix} (U_1^T)^*(z) \\ S^-(z)\end{matrix}\right]=I_m
$$
and
$$
U(z)=\left[\begin{matrix} U_1(z) \\ (S^+)^T(z)\end{matrix}\right]
$$
is the paraunitary matrix-function we wanted to find. Indeed,
clearly $U(z)$ is of size $m$ and length $N$, and we show that
\begin{equation}
\det U(z)=c\cdot z^N, \;\;\;|c|=1.
\end{equation}

Due to Lemma 2, $\det U(z)=cz^k$, $|c|=1$, for some positive
integer $k\geq N$. Hence (see (15))
\begin{equation}
\sum_{n=0}^N \overline{\alpha_{jn}}z^{-n}=u_{1j}^*(z)=c\cdot
z^{-k}\cdot\cof\big(u_{1j}(z)\big),\;\;\;j=1,2,\ldots,m.
\end{equation}
Since $S^+(0)$ is of rank $m-1$, then $
\cof\big(u_{1j}(0)\big)\not=0$ for at least one
$j\in\{1,2,\ldots,m\}$ so that the first coefficient of the
polynomial $ \cof\big(u_{1j}(z)\big)$ differs from $0$ for at
least one $j$.  Thus it follows from (22) that $k\leq N$ and hence
$k=N$, which yields (21). The desired $U(z)$ is found and let us
show its uniqueness.

Assume now that $U(z)$ is any  $m\times m$ paraunitary polynomial
matrix-function, with the first row (16), which satisfies (21),
and let $U_{m-1}(z)$ be the $(m-1)\times m$ matrix-polynomial
which is formed by deleting the first row in $U(z)$. It is obvious
that $U^T_{m-1}(z)$ is an $m\times(m-1)$ polynomial spectral
factor of (20) so that, by virtue of Theorem 1, we get
 $U^T_{m-1}(z)=S^+(z)U\Longleftrightarrow
U(z)=\left(\begin{matrix}1&0\\0&U\end{matrix}\right)
 \left(\begin{matrix} U_1(z) \\ (S^+)^T(z)\end{matrix}\right)$
immediately after  we establish that $U_{m-1}(z)$ is of full rank
$m-1$ for each $z$ inside $\mathbb{T}$. But $rank\,
U_{m-1}(z)=m-1$ for any $z\not=0$ since (21) implies that $rank\,
U(z)=m$, $z\not=0$, and $\sum_{n=0}^N
\overline{\alpha_{jn}}\,z^{-n}=u_{1j}^*(z)=c\,z^{-N}\,\cof\big(u_{1j}(z)\big)$
(see (22)), $\alpha_{jN}\not=0$, implies that
$\cof\big(u_{1j}(0)\big)\not=0$, which means that $rank\,
U_{m-1}(0)=m-1$.

\end{document}